\documentclass{article}
\usepackage[utf8]{inputenc} 
\usepackage[T1]{fontenc}
\usepackage{lipsum}
\usepackage{fullpage}
\title{Chern-Kuiper's inequalities}
\author{Diego N. Guajardo\footnote{Partially supported by CNPq and FAPERJ.}}
\usepackage{graphicx}
\usepackage{amsmath}
\usepackage{amssymb}
\usepackage{pst-node}

\usepackage{tikz-cd} 

\usepackage{bigints}
\usepackage{relsize}
\usepackage{amsthm}
\usepackage{hyperref}
\hypersetup{
    colorlinks=true,
    linkcolor=blue,
    filecolor=magenta,      
    urlcolor=cyan,
    citecolor=blue,
}

\usepackage{tikz-cd}

\makeatletter
\newcommand{\tpitchfork}{%
  \vbox{
    \baselineskip\z@skip
    \lineskip-.52ex
    \lineskiplimit\maxdimen
    \m@th
    \ialign{##\crcr\hidewidth\smash{$-$}\hidewidth\crcr$\pitchfork$\crcr}
  }%
}
\makeatother

\usepackage[sorting=nyt]{biblatex} 

\DeclareFieldFormat[article, inbook]{title}{#1} 
\DeclareFieldFormat[article]{pages}{#1}
\DeclareFieldFormat{journaltitle}{#1\isdot}
\DeclareDelimFormat[bib]{nametitledelim}{: }
\renewbibmacro{in:}{}

\renewbibmacro*{issue+date}{%
  \ifboolexpr{test {\iffieldundef{year}}
              and test {\iffieldundef{issue}}}
    {}
    {\printtext[parens]{%
       \printfield{issue}%
       \setunit*{\addspace}%
       \usebibmacro{date}}}%
  \newunit}

\usepackage[shortlabels]{enumitem}

\newtheorem{thm}{Theorem}[section]

\newtheorem{lema}[thm]{Lemma}
\newtheorem{prop}[thm]{Proposition}
\newtheorem{cor}[thm]{Corollary}

\newtheorem{defn}[thm]{Definition}

\theoremstyle{remark}
\newtheorem{remark}[thm]{Remark}

\usepackage{ifthen} 
\usepackage{xifthen} 

\DeclareMathOperator{\rank}{rank}

\newcommand{\R}     {\mathbb{R}}

\newcommand{\de}[2][]{%
    \ifthenelse{\isempty{#1}}
        {\partial_{#2}}
        {\partial_{#2}^#1}
}

\newcommand{\D}{\Delta}

\newcommand{\map}[4]{#1:#2^{#3}\rightarrow#4}
\newcommand{\inner}[2]{\langle#1,#2\rangle}

\usepackage{xcolor}

\newcommand{\eref}[1]{(\ref{#1})}
\newcommand{\tref}[1]{Theorem \ref{#1}}
\newcommand{\pref}[1]{Proposition \ref{#1}}

\newcommand{\cref}[1]{Corollary \ref{#1}}

\newcommand{\lref}[1]{Lemma \ref{#1}}
\newcommand{\sref}[1]{Section \ref{#1}}

\providecommand{\keywords}[1]
{
  \small	
  \textbf{\textit{Keywords---}} #1
}

\begin{document}

\maketitle\begin{abstract}
    Given a Euclidean submanifold $\map{g}{M}{n}{\R^{n+p}}$, Chern and Kuiper provided inequalities between $\mu$ and $\nu_g$, the ranks of the nullity of $M^n$ and the relative nullity of $g$ respectively.
    Namely, they prove that 
    \begin{equation}\label{desigualdades de Chern-Kuiper's}
        \nu_g\leq\mu\leq\nu_g+p.
    \end{equation}
    In this work, we study the submanifolds with $\nu_g\neq\mu$.
    More precisely, we characterize locally the ones with $0\neq(\mu-\nu_g)\in\{p,p-1,p-2\}$ under the hypothesis of $\nu_g\leq n-p-1$.
\end{abstract}
\hspace{30pt}
\section{Introduction}

There are two associated distributions to a submanifold $\map{g}{M}{n}{\R^{n+p}}$, the {\it nullity} $\Gamma\subseteq TM$ of the curvature tensor and the {\it relative nullity} $\D_g\subseteq TM$, i.e., the nullity of the second fundamental form $\alpha$ of $g$.
The relative nullity plays a fundamental role in many works of submanifold theory; for example \cite{DFaust}, \cite{DFTinter}, \cite{DJgbendings}, \cite{FFhyperbolen2}, and \cite{FGsingular}.
In many of them, this distribution coincides with the nullity, turning the problem into an intrinsic one; besides the ones already cited, see \cite{DFgenrigcodim2}, \cite{FZnegcurv1}, \cite{YoSC}.

We want to understand the submanifolds whose relative nullity does not coincide with the nullity.
There are two natural families of submanifolds with $\nu_g\neq\mu$. 
Firstly, if $M^n$ is flat and $g$ is not (an open subset of) an affine subspace then $\D_g\neq TM=\Gamma$.
Secondly, we have the compositions, that is, if $\map{\hat{g}}{M}{n}{\R^{n+\ell}}$ has nontrivial nullity and $\map{h}{U\subseteq\R}{n+\ell}{\R^{n+p}}$ is a flat submanifold with $\hat{g}(M^n)\subseteq U$ then generically $\map{g=h\circ\hat{g}}{M}{n}{\R^{n+p}}$ has less relative nullity, in particular $\D_g\neq\Gamma$.
Theorem 1 of \cite{FZnegcurv2} is an example of this phenomenon.

As a starting point, Gauss equation shows that $\D_g\subseteq\Gamma$.
Furthermore, Chern and Kuiper provided a complementary relation in \cite{ChKineq}. 
Namely, they showed that the ranks $\mu:=\dim(\Gamma)$ and $\nu_g:=\dim(\D_g)$ are related by \eref{desigualdades de Chern-Kuiper's}. 
Straightforward computations show that if $\nu_g=\mu-p$ then $M^n$ is flat and $\nu_g=\mu-p=n-p$. 
Proposition 7 of \cite{DFgenrigcodim2} analyzes the next case of the Chern-Kuiper's inequalities in a restricted situation.
It shows that if $\map{g}{M}{n}{\R^{n+2}}$ has $\nu_g=\mu-1=n-3$ then $g$ is locally a composition.
However, the authors' approach seems difficult to generalize.
The first result of this work extends that proposition, and the generalization is in two directions. 
We allow higher codimensions and do not impose a particular rank for the nullity.

\begin{thm}\label{teorema de composicion chern-kuiper mu=nu+p-1}
    Let $\map{g}{M}{n}{\R^{n+p}}$ be a submanifold with $p\geq 2$ and
    $$\nu_g=\mu-p+1\leq n-p-1.$$
    Then $g=G\circ\hat{g}$ is a composition, where $G:N^{n+1}\rightarrow\R^{n+p}$ is a flat submanifold and $\map{\hat{g}}{M}{n}{N^{n+1}}$ is an isometric embedding. 
    Moreover, $\D_{\hat{g}}=\Gamma$ and $\nu_G=(n+1)-(p-1)$.
\end{thm}
In particular, with \tref{teorema de composicion chern-kuiper mu=nu+p-1} we characterize locally the submanifolds $\map{g}{M}{n}{\R^{n+2}}$ with $\mu\neq\nu_g$.
Observe that the inequality condition in the last result is equivalent to $M^n$ being nowhere flat.

Using our technique, we analyze the next case of Chern-Kuiper's inequalities. 
We show that if $p\geq 3$ and $\nu_g=\mu-p+2\leq n-p-1$ then, on connected components of a dense subset of $M^n$, $g$ is also a composition.

\begin{thm}\label{thm Ch-K nu+p-2=mu}
    Let $\map{g}{M}{n}{\R^{n+p}}$ be an isometric immersion with $p\geq 3$ and
    $$\nu_g=\mu-p+2\leq n-p-1.$$
    Let $U$ be a connected component of an open dense subset of $M^n$ where $(p-\ell):=\dim(\mathcal{S}(\alpha|_{TM\times\Gamma}))$ is constant.
    Then $\ell\in\{1,2\}$ and $g|_U=G\circ\hat{g}$ is a composition, where $\map{G}{N}{n+\ell}{\R^{n+p}}$ is a flat submanifold and $\map{\hat{g}}{U\subseteq M}{n}{N^{n+\ell}}$ is  an isometric embedding. 
    Moreover, $\D_{\hat{g}}=\Gamma$ and $\nu_G\in\{(n+1)-(p-j)\}_{j=\ell}^2$.
\end{thm}

The organization of this paper is as follows.
In \sref{section preliminaries}, we recall flat bilinear forms, properties of the nullities, and ruled extensions, among others.
\sref{Ch-K section} is devoted to analyzing the submanifolds with $\nu_g\neq\mu$. 
More precisely, is divided into subsections to analyze each possible value of $\mu-\nu_g$.
Lastly, in \sref{final section}, we give some final comments about this work.

\vspace{0.4cm}

{\it Acknowledgment.} This work is a portion of the author's Ph.D. thesis at IMPA - Rio de Janeiro. The author would like to thank his adviser, Prof. Luis Florit, for his orientation. 

    
\section{Preliminaries}\label{section preliminaries}
In this section, we introduce the main techniques used in this article. Firstly, we discuss the basic properties of bilinear forms. 
Then, we analyze the two principal distributions of this work, which are the nullity and the relative nullity.
The final subsection summarizes the properties of ruled extensions.
\subsection{Flat bilinear forms}\label{seccion preliminar formas bilineares}

Given a bilinear map $\beta:\mathbb{V}\times \mathbb{U}\rightarrow \mathbb{W}$ between real vector spaces, set
\begin{equation*}
    \mathcal{S}(\beta)=\text{span}\{\beta(X,Y):X\in \mathbb{V}, Y\in \mathbb{U}\}\subseteq \mathbb{W}.
\end{equation*}
The (left) \emph{nullity} of $\beta$ is the vector subspace
\begin{equation*}
    \Delta_\beta=\{X\in \mathbb{V}:\beta(X,Y)=0\, ,\,\forall Y\in \mathbb{U}\}\subseteq \mathbb{V}.
\end{equation*}
For each $Y\in \mathbb{U}$ we denote by $\beta^Y:\mathbb{V}\rightarrow \mathbb{W}$ the linear map defined by $\beta^Y\hspace{-0.1cm}(X)=\beta(X,Y)$. 
Let
\begin{equation*}
    \text{Re}(\beta)=\{Y\in \mathbb{U}:\dim(\text{Im}(\beta^Y))\text{ is maximal}\}
\end{equation*}
be the set of (right) \emph{regular elements of} $\beta$, which is open and dense in $\mathbb{U}$. 
There are similar definitions for left regular elements and right nullity.

Assume now that $\mathbb{W}$ has a positive definite inner product $\langle\cdot,\cdot\rangle:\mathbb{W}\times \mathbb{W}\rightarrow\mathbb{R}$. 
We say that $\beta$ is $\mathit{flat}$ if
\begin{equation*}
    \langle\beta(X,Y),\beta(Z,W)\rangle=\langle\beta(X,W),\beta(Z,Y)\rangle\quad\forall X,Z\in \mathbb{V}\quad \forall Y,W\in \mathbb{U}.
\end{equation*}

The next result is due to Moore in \cite{Moo}.
It lets us determine the nullity of a flat bilinear form.
\begin{lema}\label{nulidad para no simetrica}
    Let $\beta:\mathbb{V}\times \mathbb{U}\rightarrow \mathbb{W}$ be a flat bilinear form. 
    If $Z_0\in \mathbb{U}$ is a right regular element, then $\Delta_\beta=\ker(\beta^{Z_0})$.
    In particular,
    $\dim(\Delta_\beta)=\dim(\mathbb{V})-\dim(\mathrm{Im}(\beta^{Z_0}))\geq \dim(\mathbb{V})-\dim(\mathbb{W})$.

\end{lema}



\subsection{Intrinsic and relative nullities} 
We describe now the two main distributions of this work, the nullity of a Riemannian manifold and the relative nullity of a submanifold.
\vspace{0.4cm}

Given a Riemannian manifold $M^n$ and $x\in M^n$, the {\it nullity of} $M^n$ {\it at} $x$ is the nullity of its curvature tensor $R$ at $x$, that is, the subspace of $T_xM$ given by
\begin{equation*}
    \Gamma(x)=\{X\in T_xM: R(X,Y)=0,\forall Y\in T_xM\}.
\end{equation*}
The {\it rank of} $M^n$ {\it at} $x$ is defined by $n-\mu$, where $\mu=\dim(\Gamma(x))$. 
As the results that we are looking for are of local nature and our subspaces are all either kernels or images of smooth tensor fields, without further notice we will always work on each connected component of an open dense subset of $M^n$ where all these dimensions are constant and thus all the subbundles are smooth. 
In particular, we assume that $\mu$ is constant and hence the second Bianchi identity implies that $\Gamma$ is a totally geodesic distribution, namely, $\nabla_{\Gamma}\Gamma\subseteq\Gamma$.

For an isometric immersion $g:M^n\rightarrow \R^{n+p}$ we denote by $\alpha^g:TM\times TM\rightarrow T^\perp _gM$ its second fundamental form. 
We define the {\it relative nullity of} $g$ \emph{at} $x$ as the nullity of $\alpha^g(x)$, that is, $\Delta_g(x):=\D_{\alpha^g}$. 
The \emph{rank of} $g$ is the number $n-\nu_g$, where $\nu_g=\dim(\Delta_g)$. 
Gauss equation implies that $\D_f\subseteq\Gamma$, while Codazzi equation implies that it is a totally geodesic distribution of $M^n$.

The isometric immersion $g:M^n\rightarrow\mathbb{R}^{n+p}$ is said to be $R^d${\it -ruled}, if $R^d\subseteq TM$ is a $d$-dimensional totally geodesic distribution whose leaves are mapped by $g$ onto (open subsets of) affine subspaces of $\mathbb{R}^{n+p}$. 


\subsection{Revisiting ruled extensions}\label{section revisiting ruled extensions}
Given a submanifold $\map{g}{M}{n}{\R^{n+p}}$ with $\mu\neq\nu_g\leq n-p-1$, we want to describe $g$ as a composition $G\circ\hat{g}$, where $\map{G}{N}{n+\ell}{\R^{n+p}}$ is flat as in Theorems \ref{teorema de composicion chern-kuiper mu=nu+p-1} and \ref{thm Ch-K nu+p-2=mu}.
For this, we will use ruled extensions.
The present subsection describes the basic properties of these extensions, many of which are already present in the literature; see \cite{DFGenDefSub} and \cite{DTcompii} for example.

\text{ }

In order to describe $g$ as such a composition, the first step is to find a rank $\ell$ subbundle $L=L^\ell\subseteq T^{\perp}_gM$ to be a candidate of normal bundle of $\hat{g}$. 
Then, we consider the tensor $\phi:=\phi_L:TM\times(TM\oplus L)\rightarrow L^{\perp}$ given by
\begin{equation}\label{phi en la seccion chern kuiper}
    \phi(X,v)=(\tilde{\nabla}_Xv)_{L^{\perp}},
\end{equation}
where $\tilde{\nabla}$ is the connection of $\R^{n+p}$ and the subindex denotes the orthogonal projection on the respective subspace, that is, $L^{\perp}$.
Proposition 17 of \cite{FGsingular} shows the importance of this tensor for our work.
Namely, the flatness of $\phi$ is equivalent to the local existence of an isometric immersion $\map{\hat{g}}{U\subseteq M}{n}{\R^{n+\ell}}$ whose normal bundle is $L$ (up to a parallel identification), and its second fundamental form is the orthogonal projection of $\alpha^g$ onto $L$.
However, meaningful cases also occur when $\phi$ is non-necessarily flat, as shown by \tref{Thm L=1}. 

Consider the covariant derivative of $\phi$ as
$$(\overline{\nabla}_X\phi)(Y,v):=\big(\tilde{\nabla}_X(\phi(Y,v))\big)_{L^\perp}-\phi(\nabla_XY,v)-\phi(Y,(\tilde{\nabla}_Xv)_{TM\oplus L}).$$
Notice that
\begin{align*}
    (\overline{\nabla}_X\phi)(Y,v)-(\overline{\nabla}_Y\phi)(X,v)&=\big(\tilde{\nabla}_X(\phi(Y,v)\big)_{L^\perp}-\big(\tilde{\nabla}_Y(\phi(X,v))\big)_{L^\perp}-\phi([X,Y],v)\\
    &\quad+\phi(X,(\tilde{\nabla}_Yv)_{TM\oplus L})-\phi(Y,(\tilde{\nabla}_Xv)_{TM\oplus L})\\
    &=\big(\tilde{\nabla}_X\tilde{\nabla}_Yv-\tilde{\nabla}_Y\tilde{\nabla}_Xv-\tilde{\nabla}_{[X,Y]}v\big)_{L^{\perp}},
\end{align*}
but the curvature of the ambient space is zero, so $\phi$ satisfies the following Codazzi equation
\begin{equation}\label{Codazzi phi}
    (\overline{\nabla}_X\phi)(Y,v)=(\overline{\nabla}_Y\phi)(X,v),\quad\forall X,Y\in TM,\,\forall v\in TM\oplus L.
\end{equation}

We denote by $\D_\phi^l$ and $\D_\phi^r$ the left and right nullities of $\phi$ respectively.
Certainly, $\D_\phi^l\subseteq \D_\phi^r\cap TM$.
Moreover, Codazzi equation \eref{Codazzi phi} implies that $\D_\phi^l\subseteq TM$ is integrable and $\tilde{\nabla}_{\D_\phi^l}\D_\phi^r\subseteq\D_\phi^r$.
In particular, if $\D_{\phi}^l=\D_{\phi}^r$ then $g$ is $\D_{\phi}^l$-ruled.

Given such $\phi$, we define the {\it curvature of $\phi$} as the tensor $R_\phi$ given by
$$R_\phi(X,Y,v,w)=\inner{\phi(X,w)}{\phi(Y,v)}-\inner{\phi(X,v)}{\phi(Y,w)},\quad \forall X,Y\in TM,\,\forall v,w\in TM\oplus L.$$
In particular, $\phi$ is flat when its curvature is zero.
Intuitively, $\phi$ and $R_\phi$ are the second fundamental form and curvature of the extension respectively.

The curvature of $\phi$ satisfies the following Bianchi identities.

\begin{lema}[Bianchi's identities]\label{Lemma Bianchi phi}
    The curvature of $\phi$ satisfies the following first and second Bianchi identities
    \begin{equation}\label{1 Bianchi identity phi}
        \sum R_\phi(S,T,U,v)=0,\quad\forall S,T,U\in TM,\,\forall v\in TM\oplus L,
    \end{equation}
    \begin{equation}\label{Bianchi phi}
        \sum(\nabla_SR_\phi)(T,U,v,w)=0,\quad \forall S,T,U\in TM,\,\forall v,w\in TM\oplus L.
    \end{equation}
     where the sum denotes the cyclic sum over $S, T,$ and $U$. 
\end{lema}
\begin{proof}
    The first Bianchi identity comes from opening the curvatures and simplifying terms.

    To prove the second identity, we make the computations at a fixed point $q\in M^n$, and we take smooth sections such that the derivatives between $S, T, U, v$, and $w$ are zero at $q$, that is, $(\nabla_ST)(q)=0$, $(\tilde{\nabla}_Sv)_{TM\oplus L}(q)=0$, and so on.
    Denote by $B$ the left-hand side of \eref{Bianchi phi}, and notice that 
    \begin{align*}
        B&=\sum S\big(R_\phi(T,U,v,w)\big)\\
        &=\sum\Big[\inner{(\overline{\nabla}_S\phi)(T,w)}{\phi(U,v)}\hspace{-0.5mm}+\hspace{-0.5mm}\inner{\phi(T,w)}{(\overline{\nabla}_S\phi)(U,v)}\Big]\hspace{-0.5mm}-\hspace{-0.5mm}\sum\Big[\inner{(\overline{\nabla}_S\phi)(T,v)}{\phi(U,w)}\hspace{-0.5mm}+\hspace{-0.5mm}\inner{\phi(T,v)}{(\overline{\nabla}_S\phi)(U,w)}\Big]\\
        &=\sum\Big[\inner{(\overline{\nabla}_S\phi)(T,w)}{\phi(U,v)}\hspace{-0.6mm}-\hspace{-0.6mm}\inner{\phi(T,v)}{(\overline{\nabla}_S\phi)(U,w)}\Big]\hspace{-0.5mm}+\hspace{-0.5mm}\sum\Big[\inner{\phi(T,w)}{(\overline{\nabla}_S\phi)(U,v)}\hspace{-0.6mm}-\hspace{-0.6mm}\inner{(\overline{\nabla}_S\phi)(T,v)}{\phi(U,w)}\Big]\hspace{-0.2mm},
    \end{align*}
    rearranging the terms of both sums we get
    $$B=\sum\big[\inner{(\overline{\nabla}_S\phi)(T,w)-(\overline{\nabla}_T\phi)(S,w)}{\phi(U,v)}\big]+\sum\big[\inner{\phi(T,w)}{(\overline{\nabla}_S\phi)(U,v)-(\overline{\nabla}_U\phi)(S,v)}\big],$$
    which is zero since $\phi$ satisfies Codazzi equation \eref{Codazzi phi}.
\end{proof}

We denote by $\Gamma_{\phi}^l$ and $\Gamma_{\phi}^r$ the left and right nullities of $R_\phi$, that is
$$\Gamma_\phi^l:=\{X\in TM|\, R_\phi(X,TM,TM\oplus L,TM\oplus L)=0\}\subseteq TM,$$
and
$$\Gamma_\phi^r:=\{v\in TM\oplus L|\, R_\phi(TM,TM,TM\oplus L,v)=0\}\subseteq TM\oplus L.$$
Certainly, $\D_\phi^l\subseteq\Gamma_\phi^l$ and $\D_\phi^r\subseteq\Gamma_\phi^r$.
The first Bianchi identity \eref{1 Bianchi identity phi} shows that $\Gamma_\phi^l\subseteq\Gamma_\phi^r\cap TM$.
Moreover, the second one implies that $\Gamma_\phi^l\subseteq TM$ is integrable and $\big(\tilde{\nabla}_{\Gamma_\phi^l}\Gamma_\phi^r\big)_{TM\oplus L}\subseteq \Gamma_\phi^r$.
In particular, $\tilde{\nabla}_{\D_\phi^l}\Gamma_\phi^r\subseteq\Gamma_\phi^r$.

Consider the vector bundle $\Lambda:=\D_\phi^r\cap(\D_\phi^l)^{\perp}\subseteq TM\oplus L$, and suppose that $\text{rank}(\Lambda)=\ell=\text{rank}(L)$.
The {\it ruled extension} $G:\Lambda\rightarrow\R^{n+p}$ of $g$ is given by
$$G(\xi_q)=g(p)+\xi_q,\quad\forall q\in M^n,\,\forall \xi_q\in\Lambda_q$$
We restrict $G$ to a neighborhood $N^{n+\ell}$ of the zero section $\map{\hat{g}}{M}{n}{N^{n+\ell}}\subseteq\Lambda$ in order to $G$ being an immersion.
Moreover, we endow $N^{n+\ell}$ with the induced metric by $G$.

\begin{prop}\label{prop ruled extensions}
    Assume that $\D_\phi^l=\D_\phi^r\cap TM$ and $\rank(\Lambda)=\ell=\rank(L)$.
    Then $\D_\phi^r$ is the nullity of $G$, that is, $\D_G=\D_\phi^r$ up to a parallel identification along $\D_\phi^l$. 
    Similarly, the nullity of $N^{n+\ell}$ is given by $\Gamma_\phi^r$.
\end{prop}
\begin{proof}
    First, $G$ is $\D_{\phi}^r$-ruled since $\tilde{\nabla}_{\D_\phi^l}\D_\phi^r\subseteq\D_{\phi}^r$.
    Take a section $\xi$ of $N^{n+\ell}\subseteq\Lambda$ and $Y\in TM$, then 
    $$G_*(\xi_*Y)=g_*Y+\tilde{\nabla}_Y\xi\in TM\oplus L=G_*(TN).$$ 
    Notice that $TM\oplus L$ is parallel along $\D_\phi^r$ since $\D_\phi^l=\D_\phi^r\cap TM$, so $\Lambda\subseteq\D_G$.
    As $TN\cong TM\oplus \Lambda$, to compute the second fundamental form of $G$ is enough to understand $\alpha^G|_{TM\times(TM\oplus L)}$.
    If $X\in TM$ then
    $$\tilde{\nabla}_X\big(G_*(\xi_*Y)\big)=g_*\nabla_XY+\alpha(X,Y)+\tilde{\nabla}_X\tilde{\nabla}_Y\xi,$$
    so 
    $$\alpha^G(X,\xi_*Y)=\big(\tilde{\nabla}_X\big(G_*(\xi_*Y)\big)\big)_{L^{\perp}}=\big(\alpha(X,Y)\big)_{L^\perp}+\big(\tilde{\nabla}_X\tilde{\nabla}_Y\xi\big)_{L^\perp}=\phi(X,Y)+\phi(X,\tilde{\nabla}_Y\xi)=\phi(X,\xi_*Y),$$ 
    up to parallel identifications. 
    This proves that $\D_\phi^r=\D_G$.

    Finally, Gauss equation shows that the curvature tensor $R_N$ of $N^{n+\ell}$ is given by 
    $$R_N(X,Y,v,w)=\inner{\alpha^G(X,w)}{\alpha^G(Y,v)}-\inner{\alpha^G(X,v)}{\alpha^G(Y,w)}=R_\phi(X,Y,v,w),\quad\forall X,Y\in TM,\,\forall v,w\in TM\oplus L,$$
    which shows that $\Gamma_\phi^r$ is the nullity of $N^{n+\ell}$ since the remaining values of $R_N$ involve terms of relative nullity.    
\end{proof}
\begin{remark}
    We can give a weaker version of the last proposition for $0\leq \rank(\Lambda)<\text{rank}(L)$. 
    In that case, there is an orthogonal decomposition $T_G^{\perp}N=\mathcal{L}\oplus E$ such that $\text{rank}(\mathcal{L})=\text{rank}(L)-\rank(\Lambda)$, $G$ is $\D_{\phi}^r$-ruled, and this distribution coincides with the nullity of the $E$-component of $\alpha^G$.
\end{remark}

\section{Chern-Kuiper's inequalities}\label{Ch-K section}
In this section, we describe the basic properties of the submanifolds $\map{g}{M}{n}{\R^{n+p}}$ whose relative nullity $\D_g$ does not coincide with the intrinsic nullity $\Gamma$.
In the following subsections, we analyze the cases $\nu_g=\mu-p$, $\nu_g=\mu-p+1$, and $\nu_g=\mu-p+2$ respectively.

\text{ }

Let $\map{g}{M}{n}{\R^{n+p}}$ be a submanifold with non-trivial intrinsic nullity $\Gamma\neq0$. 
Call $\alpha$ its second fundamental form and $\D_g$ its relative nullity. 
Gauss equation implies that $\D_g\subseteq\Gamma$ and the flatness of the bilinear tensor $\beta:=\alpha|_{TM\times\Gamma}$.
Let $\D_{\beta}$ be the (left) nullity of $\beta$.
The flatness of $\beta$ implies that
\begin{equation}\label{alpha(Y,X) in S(beta)perp}
    \alpha(Y,X)\in\mathcal{S}(\beta)^{\perp},\quad\forall Y\in\D_\beta,\, \forall X\in TM.
\end{equation}
So in particular
$$\alpha(Y,X)\in\mathcal{S}(\beta)\cap\mathcal{S}(\beta)^{\perp}=0,\quad \forall Y\in\D_{\beta}\cap\Gamma,\,\forall X\in TM,$$
which shows that $\D_g=\D_{\beta}\cap\Gamma$.
Then, we have the following relation 
\begin{equation}\label{suma de dimensiones=suma de dimensiones}
    \nu_g+\dim(\D_{\beta}+\Gamma)=\dim(\D_{\beta})+\mu.
\end{equation}

Notice that $\D_\beta\subseteq TM$ is an integrable distribution. 
Indeed, Codazzi equation for $T_1,T_2\in\D_\beta$ gives
$$\alpha([T_1,T_2],Z)=\alpha(T_1,\nabla_{T_2}Z)-\alpha(T_2,\nabla_{T_1}Z),\quad\forall Z\in\Gamma,$$
but the left-hand side belongs to $\mathcal{S}(\beta)$ and the right-hand side to $\mathcal{S}(\beta)^{\perp}$ by \eref{alpha(Y,X) in S(beta)perp}, so $[T_1,T_2]\in\D_\beta$.

Let us recall Chern-Kuiper's inequalities, and provide a quick proof. 
\begin{prop}[Chern-Kuiper's inequalities \cite{ChKineq}]\label{teorema desigualdades chern Kuiper}
    Let $\map{g}{M}{n}{\R^{n+p}}$ be a submanifold, then \eref{desigualdades de Chern-Kuiper's} holds.
\end{prop}
\begin{proof}
    As $\D_g\subseteq\Gamma$ then $\nu_g\leq\mu$.
    Take $Z_0\in\text{Re}(\beta)\subseteq\Gamma$ a (right) regular element of $\beta$, then by \lref{nulidad para no simetrica} and \eref{suma de dimensiones=suma de dimensiones} we get that
    \begin{equation}\label{hhhhh}
        \nu_g+n\geq\nu_g+\dim(\D_\beta+\Gamma)=\dim(\D_\beta)+\mu=n-\dim(\text{Im}(\beta^{Z_0}))+\mu\geq n-p+\mu.
    \end{equation}
    which proves the second inequality of \eref{desigualdades de Chern-Kuiper's}.
\end{proof}

Before analyzing the inequality cases of the Chern-Kuiper's inequalities, we present a result that gives us bounds for the rank of $\mathcal{S}(\beta)$ under the hypothesis of $\nu_g\leq n-p-1$.
\begin{lema}\label{lemma bound of L}
    Let $\map{g}{M}{n}{\R^{n+p}}$ be a submanifold with $\nu_g\leq n-p-1$.
    Then
    \begin{equation}\label{eq bound of the rank}
        \mu-\nu_g\leq\dim\big(\mathcal{S}(\beta)\big)\leq p-1.
    \end{equation}
\end{lema}
\begin{proof}
    The first inequality comes from \eref{suma de dimensiones=suma de dimensiones} and \lref{nulidad para no simetrica} since
    $$n+\nu_g\geq\dim(\D_\beta+\Gamma)+\nu_g=\dim(\D_\beta)+\mu\geq n-\dim\big(\mathcal{S}(\beta)\big)+\mu.$$
    
    On the other hand, suppose by contradiction that $\mathcal{S}(\beta)=T^\perp_{g}M$. Then, by
    \eref{alpha(Y,X) in S(beta)perp}, we have that $\D_\beta=\D_g$. 
    However, in this case, \lref{nulidad para no simetrica} implies that
    $$\nu_g=\dim(\D_\beta)\geq n-\dim\big(\mathcal{S}(\beta)\big)=n-p,$$
    which is absurd.
\end{proof}

\subsection{The case \texorpdfstring{$\nu_g=\mu-p$}{TEXT}}
In this subsection, we analyze the maximal case of Chern-Kuiper's inequalities. 
We also describe the technique that will be used for the following cases.

\text{ }

The next result shows that only flat submanifolds attain the second inequality of \eref{desigualdades de Chern-Kuiper's}.

\begin{prop}\label{proposicion caso extremo ChK es flat}
    Let $\map{g}{M}{n}{\R^{n+p}}$ be a submanifold with \begin{equation}\label{igualdad extrema de chern-kuiper, planas}
        \mu=\nu_g+p.
    \end{equation}
    Then $M^n$ is flat, in particular $\mu=\nu_g+p=n$.
\end{prop}
\begin{proof}
    In this case we must have equalities in \eref{hhhhh}, hence $\D_{\beta}+\Gamma=TM$ and $\text{Im}(\beta^{Z_0})=\mathcal{S}(\beta)=T^{\perp}_gM$. 
   Then \eref{alpha(Y,X) in S(beta)perp} implies that $\D_{\beta}=\D_g\subseteq\Gamma$, and thus $\Gamma=\D_{\beta}+\Gamma=TM$. 
\end{proof}
\begin{remark}
There are natural parametrizations for flat submanifolds attaining \eref{igualdad extrema de chern-kuiper, planas}; see \cite{DGgaussP} for $p=1$ and \cite{FFhyperbolen2} for $p=2$.
This is generalized in \cite{YoFlatE} for any $p\leq n$.
\end{remark}

Chern-Kuiper's inequalities and \pref{proposicion caso extremo ChK es flat} characterize the hypersurfaces with $\D_g\neq\Gamma$ by means of the Gauss parametrization. 
Hence, we assume from now on that $p\geq 2$.

There is a natural way to produce submanifolds $\map{g}{M}{n}{\R^{n+p}}$ with $\D_g\neq\Gamma$ using compositions. 
Consider a submanifold $\map{\hat{g}}{M}{n}{\R^{n+\ell}}$ with $\Gamma=\D_{\hat{g}}\neq0$, $\ell<p$, and let $\map{G}{U\subseteq\R}{n+\ell}{\R^{n+p}}$ be an isometric immersion of an open subset $U$ of $\R^{n+\ell}$ with $\hat{g}(M^{n})\subseteq U$.
Then $g:=G\circ\hat{g}$ generically has less nullity than $\hat{g}$, so $\D_g\neq\Gamma$. 
Conversely, we will use the following strategy to prove that such a $g$ must be a composition.
Naively, $\mathcal{S}(\beta)$ should be $T^{\perp}_jU$ (or, at least, contained), and so $L:=\mathcal{S}(\beta)^{\perp}\subseteq T^{\perp}_gM$ is a candidate to be $T^{\perp}_{\hat{g}}M$.
Hence, we can use the techniques of \sref{section revisiting ruled extensions}.
Namely, we will study the properties of the tensor $\phi=\phi_L$ given by \eref{phi en la seccion chern kuiper} associated with $L$, then we will use \pref{prop ruled extensions} to obtain the desired composition.

\subsection{The case \texorpdfstring{$\nu_g=\mu-p+1$}{TEXT}}
This subsection is dedicated to analyzing the following case of Chern-Kuiper's inequalities.
We will prove a more general statement.
We characterize the submanifolds such that the first inequality of \lref{lemma bound of L} is attained; they are all flat compositions.

\text{ }

Suppose that $\map{g}{M}{n}{\R^{n+p}}$ is a submanifold with $\mu=\nu_g+p-1$ and $p\geq2$.
\lref{lemma bound of L} implies that
$$\dim\big(\mathcal{S}(\beta)\big)=\mu-\nu_g=p-1.$$ 
In particular, we are in a situation where the lower bound of \eref{eq bound of the rank} is attained.
The following result analyzes this equality in complete generality.
\tref{teorema de composicion chern-kuiper mu=nu+p-1} is a direct consequence of it.

\begin{thm}\label{teo de composicion para betaD}
     Consider a submanifold $\map{g}{M}{n}{\R^{n+p}}$ with $\D_g\neq\Gamma$.
     Let $\beta=\alpha|_{TM\times\Gamma}$ and suppose that
     $$p-\ell:=\dim\big(\mathcal{S}(\beta)\big)=\mu-\nu_g<p.$$ 
     Then $g=G\circ\hat{g}$ is a composition, where $\map{G}{N}{n+\ell}{\R^{n+p}}$ is a flat submanifold and $\map{\hat{g}}{M}{n}{N^{n+\ell}}$ is an isometric embedding.
     Moreover, $\D_{\hat{g}}=\Gamma$ and $\nu_G=(n+1)-(p-\ell)$. 
\end{thm}
\begin{proof}
    Let $Z_0\in \Gamma$ be a (right) regular value of $\beta$.
    \lref{nulidad para no simetrica} and \eref{suma de dimensiones=suma de dimensiones} imply that
    $$\dim(\D_\beta+\Gamma)=\dim(\D_\beta)+\dim(\Gamma)-\dim(\D_g)= n-\dim(\text{Im}(\beta^{Z_0}))+\mu-\nu_g\geq n,$$
    which shows that $\alpha(Z_0,TM)=L^{\perp}$ and $\D_\beta+\Gamma=TM$.
    In particular, $\mathcal{S}(\beta)=\mathcal{S}(\alpha|_{\Gamma\times \Gamma})$.

    Let $L:=\mathcal{S}(\beta)^\perp\subseteq T^{\perp}_gM$ and consider the tensor $\phi=\phi_L$ given by \eref{phi en la seccion chern kuiper}.
    We will use \pref{prop ruled extensions} to prove that $g$ is such a composition.
    Hence, we need to show that $\phi$ is flat, $$\dim(\D_{\phi}^r)=\dim(\D_{\phi}^l)+\ell=n+\ell-(p-\ell),$$ 
    and
    \begin{equation}\label{hjhj}
        \D_\beta=\D_\phi^l=\D_\phi^r\cap TM.
    \end{equation}
    
    Notice that $\phi(\D_\beta,TM)=0$ by \eref{alpha(Y,X) in S(beta)perp}, and so $\D_\beta=\D_\phi^r\cap TM$. 
    Moreover, if $Y\in\D_\beta$ then Codazzi equation for $\xi\in L$ and $Z_1,Z_2\in \Gamma$ gives us that
    $$\inner{\phi(Y,\xi)}{\alpha(Z_1,Z_2)}=\inner{\nabla^{\perp}_Y\xi}{\alpha(Z_1,Z_2)}=-\inner{\xi}{(\nabla_Y^{\perp}\alpha)(Z_1,Z_2)}=\inner{\xi}{\alpha(Y,\nabla_{Z_1}Z_2)}=0,\quad \forall Z_1,Z_2\in\Gamma,$$
    since $\Gamma\subseteq TM$ is totally geodesic.
    Hence, $\phi(Y,\xi)=0$  since $\mathcal{S}(\alpha|_{\Gamma\times\Gamma})=\mathcal{S}(\beta)=L^{\perp}$, and so $\D_\beta=\D_\phi^l$. 
    This proves \eref{hjhj}.

    As $TM=\D_\beta+\Gamma$ and \eref{hjhj} holds, the flatness of $\phi$ is equivalent to the flatness of $\phi|_{\Gamma\times(\Gamma\oplus L)}$.
    Notice that $\phi|_{\Gamma\times \Gamma}=\alpha|_{\Gamma\times \Gamma}$ is flat by Gauss equation.
    On the other hand, if $Z_1,Z_2,Z_3\in \Gamma$ and $\xi\in L$ then
    $$\inner{\phi(Z_1,\xi)}{\phi(Z_2,Z_3)}=\inner{\nabla^{\perp}_{Z_1}\xi}{\alpha(Z_2,Z_3)}=-\inner{\xi}{(\nabla^{\perp}_{Z_1}\alpha)(Z_2,Z_3)},$$
    which is symmetric in $Z_1$ and $Z_2$ by Codazzi equation. 
    Hence, to prove the flatness is enough to show that 
    \begin{equation}\label{jhjh}
        \inner{\phi(T_1,\xi_1)}{\phi(t   T_2,\xi_2)}=\inner{\phi(T_1,\xi_2)}{\phi(T_2,\xi_1)},\quad\forall T_1,T_2\in\Gamma,\,\forall\xi_1,\xi_2\in L.
    \end{equation}
    Notice first that the nullity of $\alpha|_{\Gamma\times\Gamma}$ is $\D_\beta\cap\Gamma=\D_g$. 
    Thus, $\alpha|_{\Gamma\times\Gamma}$ is completely described by Theorem 2 of \cite{Moo}. 
    Namely, there are vectors $Z_1,\ldots,Z_{p-\ell}\in \Gamma\cap\D_g^{\perp}$ such that $\alpha(Z_i,Z_j)=0$ for $i\neq j$ and the set
    $\{\rho_i:=\alpha(Z_i,Z_i)\}_{i=1}^{p-\ell}$
    is an orthonormal basis of $L^{\perp}$. 
    Given $\xi\in L$, Codazzi equation implies that
    $$\inner{\phi(Z_i,\xi)}{\rho_j}=-\inner{\xi}{(\nabla^{\perp}_{Z_i}\alpha)(Z_j,Z_j)}=\inner{\xi}{\nabla^{\perp}_{Z_j}(\alpha(Z_i,Z_j))-\alpha(\nabla_{Z_j}Z_i,Z_j)-\alpha(Z_i,\nabla_{Z_j}Z_j)}=0,\quad\forall i\neq j.$$
    Then $\phi(Z_i,\xi)=\lambda_i(\xi)\rho_i$ for some 1-forms $\lambda_i:L\rightarrow \R$. 
    Then \eref{jhjh} holds since $\{Z_1,\ldots,Z_{p-\ell}\}$ is a basis of $\Gamma$ and
    $$\inner{\phi(Z_i,\xi_1)}{\phi(Z_j,\xi_2)}=\delta_{ij}\lambda_i(\xi_1)\lambda_j(\xi_2),\quad\forall i,j,\,\forall \xi_1,\xi_2\in L.$$
    
    Finally, by \lref{nulidad para no simetrica}, we have for $Z_0\in\Gamma$ a regular element of $\beta$ that
    $$\dim(\D_{\phi}^r)=n+\ell-\dim\text{Im}(\phi^{Z_0})=(n+\ell)-(p-\ell)=n-\dim\text{Im}(\beta^{Z_0})+\ell=\dim(\D_\beta)+\ell=\dim(\D_\phi^l)+\ell.$$
    
    The result now follows from \pref{prop ruled extensions}. 
    Notice that the second fundamental form of $\hat{g}$ is the orthogonal projection of $\alpha$ onto $L$, but as $\alpha(\Gamma,TM)\in L^{\perp}$ then $\Gamma=\D_{\hat{g}}$.
\end{proof}


We can describe locally all the submanifolds $\map{g}{M}{n}{\R^{n+2}}$ with $\D_g\neq\Gamma$. 

\begin{prop}\label{ChK mu=nu+1 en codimension 2}
    Let $\map{g}{M}{n}{\R^{n+2}}$ be a submanifold with  $\Gamma\neq\D_g$. 
    Then, on each connected component $U$ of an open dense subset of $M^n$, we have one of the following possibilities:
    \begin{enumerate}[$i)$]
        \item $\mu=\nu_g+1$ and $g|_U=j\circ\hat{g}$ is a composition where $\hat{g}:U\rightarrow V\subseteq\R^{n+1}$ and $j:V\rightarrow\R^{n+2}$ are isometric immersions with $\Gamma=\D_{\hat{g}}$;
        \item $\mu=\nu_g+2$ and $U$ is flat.
    \end{enumerate}
\end{prop}
\begin{proof}[Proof of \tref{ChK mu=nu+1 en codimension 2}]
    By \pref{proposicion caso extremo ChK es flat}, and \tref{teorema de composicion chern-kuiper mu=nu+p-1}, it only remains to analyze the case $\mu=n=\nu_g+1$.
    However, this case is a direct consequence of \tref{teo de composicion para betaD}.
\end{proof}
\begin{remark}
    Each case of \tref{ChK mu=nu+1 en codimension 2} is naturally parametrizable. 
    For $(i)$ we use the Gauss parametrization described in \cite{DGgaussP}, and Corollary 18 of \cite{FFhyperbolen2} describes the second case.
\end{remark}

\subsection{The case \texorpdfstring{$\nu_g=\mu-p+2$}{TEXT}}
In this final subsection, we discuss the next case of Chern-Kuiper's inequalities.
For this, we prove \tref{Thm L=1} which analyzes in generality the case $\ell=1$. 
This result and \tref{teo de composicion para betaD} imply \tref{thm Ch-K nu+p-2=mu}.

\text{ }

\tref{teo de composicion para betaD} describes the submanifolds that attain the first inequality of \eref{eq bound of the rank}. 
We now analyze when the second one does.
Namely, let us consider $\map{g}{M}{n}{\R^{n+p}}$ a submanifold with $\D_g\neq\Gamma$ and suppose that $L:=\mathcal{S}(\beta)^{\perp}$ has rank $\ell=1$. 
As before, consider $\phi=\phi_L$ the tensor given by \eref{phi en la seccion chern kuiper}. 
We begin with the next result.

\begin{lema}\label{lemma L=1, phi(delta-beta,TM+L)=0}
    If $L$ has rank $1$ and
    $\nu_g\leq n-p-1$ then $\D_\beta=\D_\phi^l=\D_\phi^r\cap TM$.
    Furthermore, if $\alpha(\D_\beta,\D_\beta)\neq 0$ then $\Gamma\subseteq\Gamma_\phi^l$.
\end{lema}
\begin{proof}
    On the second Bianchi identity \eref{Bianchi phi}, take $S=Z\in\Gamma$, $T=d_1,v=d_2\in\D_\beta$, and $w\in TM$ to obtain
    \begin{equation}\label{triop}
        0=R_\phi(Z,d_1,\alpha(U,d_2),w)+R_\phi(U,Z,\alpha(d_1,d_2),w)\quad\forall Z\in\Gamma,\,\forall d_1,d_2\in\Gamma,\,\forall w\in TM.
    \end{equation}
    In the last equation, fix $d_1$ and choose $0\neq d_2\in\D_\beta\cap\D_g^\perp$ such that $\hat{\alpha}(d_1,d_2)=0$.
    This is possible since $\ell=1$ and 
    $$\dim(\D_\beta\cap\D_g^\perp)=\dim(\D_\beta)-\dim(\D_g)\geq n-(p-1)-(n-p-1)=2,$$
    where the last inequality comes from \lref{nulidad para no simetrica}.     
    Let $\rho\in L$ be a fixed unit generator of $L$ and take $U\in TM$ such that $\rho=\hat{\alpha}(U,d_2)$ to obtain 
    $$0=R_\phi(Z,d_1,\rho,w)=\inner{\phi(Z,w)}{\phi(d_1,\rho)},\quad\forall Z\in\Gamma,\,\forall d_1\in\D_\beta,\,\forall w\in TM,$$
    but $\phi(\Gamma,TM)=\beta(TM,\Gamma)=L^{\perp}$, so $\phi(d_1,\rho)=0$ for any $d_1\in\D_\beta$. 
    Thus
    $$\D_\beta\subseteq\D_\phi^l\subseteq\D_\phi^r\cap TM\subseteq\D_\beta.$$

    Finally, suppose that $\alpha(\D_\beta,\D_\beta)\neq0$. 
    Then $\alpha(\D_\beta,\D_\beta)=L$ by \eref{alpha(Y,X) in S(beta)perp}. 
    Take $d_1,d_2\in\D_\beta$ such that $\alpha(d_1,d_2)=\rho$, and use them in \eref{triop} to obtain
    $$0=R_\phi(Z,U,\rho,w),\quad\forall Z\in\Gamma,\,\forall U,w\in TM,$$
    which proves that $\Gamma\subseteq\Gamma_\phi^l$ since $R_\phi(\Gamma,TM,TM,TM)=0$ by Gauss equation.
\end{proof}
\begin{remark}
    \lref{lemma L=1, phi(delta-beta,TM+L)=0} holds under the weaker assumption of $\dim(\D_\beta\cap\D_g^{\perp})\geq2$ instead of $\nu_g\leq n-p-1$.
\end{remark}

\begin{thm}\label{Thm L=1}
    Let $\map{g}{M}{n}{\R^{n+p}}$ be an isometric immersion with
    $$\mu\neq\mu_g\leq n-p-1.$$
    Suppose that $L:=\mathcal{S}(\alpha|_{TM\times\Gamma})^\perp\subseteq T^{\perp}_gM$ has rank $1$. 
    Then, on each connected component $U$ of an open dense subset of $M^n$ where $\mu$, $\nu_g$, and $k=\dim\alpha(\D_\beta,\D_\beta)$ are constant, we have the following possibilities
    \begin{enumerate}[$i)$]
        \item $k=1$ and $g|_U$ is a composition of a ruled extension $\map{G}{N}{n+1}{\R^{n+p}}$ and an isometric embedding $\hat{g}:U\subseteq M^n\rightarrow N^{n+1}$.
        Moreover, $\D_{\hat{g}}=\Gamma$,
        $$(n+1)-(p-1)\leq\nu_G\leq (n+1)-(\mu-\nu_g),$$
        and $\hat{g}_*(\Gamma)\subseteq\hat{\Gamma}$, where $\hat{\Gamma}\subseteq TN$ is the nullity of $N^{n+1}$ and satisfies that $\dim(\hat{\Gamma})\geq\mu-\nu_g+\nu_G$;
        \item $k=0$ and $g$ is $\D_\beta$-ruled.
        Moreover, the rank of the ruling is at least $(n-p+1)$. 
    \end{enumerate}
\end{thm}
\begin{proof}
    We want to use \pref{prop ruled extensions} to prove this result. 
    By \lref{lemma L=1, phi(delta-beta,TM+L)=0} we know that $\D_\beta=\D_\phi^l=\D_\phi^r\cap TM$.
    
    Suppose first that $k=1$, and so $\Gamma\subseteq\Gamma_\phi^l$ by \lref{lemma L=1, phi(delta-beta,TM+L)=0}.
    This implies that the tensor $\hat{\beta}:(TM\oplus L)\times\Gamma\rightarrow L^{\perp}$ given by $\hat{\beta}(v,Z)=\phi(Z,v)$ is a flat extension of $\beta$.
    Notice that the left nullity of $\hat{\beta}$ coincides with $\D_\phi^r$. 
    Indeed, let us verify the non-trivial contention.
    Take $v_0\in TM\oplus L$ such that $\hat{\beta}(\Gamma,v_0)=0$.
    Then, as $\Gamma\subseteq\Gamma_\phi^l$, we have that
    $$0=R_\phi(Z,X,v_0,w)=\inner{\phi(X,v_0)}{\phi(Z,w)},\quad\forall Z\in\Gamma,\,\forall X,w\in TM,$$
    but $\phi(\Gamma,TM)=L^{\perp}$, and so $v_0\in\D_\phi^r$.
    In particular, \lref{nulidad para no simetrica} for $\beta$ and $\hat{\beta}$ shows that the dimensions of $\D_\phi^r$ and $\D_\phi^l=\D_\beta$ differ by at most $1$. 
    However, if $\D_\beta=\D_\phi^l=\D_\phi^r$ then $g$ would be $\D_\beta$-ruled which is absurd since $k=1$.
    \pref{prop ruled extensions} shows that $g$ has a ruled extension $\map{G}{N}{n+1}{\R^{n+p}}$. 
    Moreover, \lref{nulidad para no simetrica} implies that the nullity of $G$ satisfies that
    $$\nu_G=\dim(\D_\phi^r)=\dim(\D_\beta)+1\geq n+1-(p-1).$$
    On the other hand, using \eref{suma de dimensiones=suma de dimensiones} we get that 
    $$\nu_G=\dim(\D_\beta)+1=\dim(\D_\beta+\Gamma)+\nu_g+1-\mu\leq n+1-(\mu-\nu_g).$$
    The bound on the nullity of $N^{n+1}$ follows from \eref{suma de dimensiones=suma de dimensiones} and $\Gamma\subseteq\Gamma_\phi^l \subseteq\Gamma_\phi^r$ since
    $$\dim(\hat{\Gamma})=\dim(\Gamma_\phi^r)\geq\dim(\D_\phi^r+\Gamma)=1+\dim(\D_\beta+\Gamma)=1+\dim(\D_\beta)+\mu-\nu_g=\nu_G+\mu-\nu_g.$$

    Finally, suppose that $k=0$, that is, $\alpha(\D_\beta,\D_\beta)=0$.
    Codazzi equation \eref{phi en la seccion chern kuiper} for $d_1,d_2\in\D_\beta$ give us that
    $$0=(\overline{\nabla}_X\phi)(d_1,d_2)-(\overline{\nabla}_{d_1}\phi)(X,d_2)=\phi(X,\nabla_{d_1}d_2),\quad\forall X\in TM,\,\forall d_1,d_2\in\D_\beta,$$
    which proves that $\D_\beta$ is totally geodesic since $\D_{\phi}^r\cap TM=\D_\beta$.
    Hence, $g$ is $\D_\beta$-ruled, and
    \lref{nulidad para no simetrica} gives the desired bound on the rank of the rulings.
\end{proof}

We prove now \tref{thm Ch-K nu+p-2=mu}.

\begin{proof}[Proof of Theorem \ref{thm Ch-K nu+p-2=mu}]
    By \lref{lemma bound of L} we know that $\ell\in\{1,2\}$.
    The case  $\ell=2$ follows from \tref{teo de composicion para betaD}.
    Assume now that $\ell=1$, so we can apply \tref{Thm L=1}.
    Then, if $k=1$, $N^{n+1}$ must be flat since 
    $$\dim(\hat{\Gamma})\geq p-2+\nu_G\geq p-2+(n+1)-(p-1)=n,$$
    but $\dim(\hat{\Gamma})=n$ is not possible by the symmetries of the curvature tensor, so $\hat{\Gamma}=TN$.
    It remains to exclude the second possibility of that result, that is, $k=0$. 
    Suppose, by contradiction, that $\ell=1$ and $g$ is $\D_g$-ruled on an open subset of $M^n$.
    Notice that as $\phi|_{TM\times\Gamma}=\beta$ is flat and $\D_\beta=\D_\phi^r\cap TM$, so $\phi|_{TM\times (\D_\beta+\Gamma)}$ is flat.
    However, by \eref{suma de dimensiones=suma de dimensiones} and \lref{nulidad para no simetrica} we know that
    $$\dim(\D_\beta+\Gamma)=\dim(\D_\beta)+\mu-\nu_g\geq n-(p-1)+p-2=n-1,$$
    so $\phi|_{TM\times TM}$ must be flat.
    Then, fixing a unit generator $\rho$ of $L$, the shape operator $A=A_\rho$ satisfies Gauss equation.
    However, as $g$ is $\D_\beta$-ruled and $A\Gamma=0$, we have that 
    $$\inner{A(\D_\beta+\Gamma)}{\D_\beta+\Gamma}=0,$$
    which implies that $\mu=\ker(A)\geq n-2$.
    This is a contradiction since $\mu=\nu_g+p-2\leq (n-p-1)+(p-2)= n-3$.
\end{proof}

\section{Final comments}\label{final section}
In this final section, we give some observations of this work.

\text{ }

The results of this work suggest that there are at least two distinct families of submanifolds $\map{g}{M}{n}{\R^{n+p}}$ with $\nu_g\neq\mu$. 
The submanifolds with rank greater than their codimension and those that do not.
Moreover, Theorems \ref{teo de composicion para betaD} and \ref{Thm L=1} suggest that, aside from the ruled cases, any submanifold of the first class is contained in a submanifold of the second one.

In the submanifold theory, there are many works in which it is necessary to exclude compositions of the form $g=G\circ\hat{g}$ where $\hat{g}:N^{n+\ell}\rightarrow\R^{n+p}$ is a flat submanifold; see \cite{DFThyperEn2} and \cite{DTcompii} for example.
Moreover, the notion of {\it honest deformation} is to exclude this type of behavior in the deformation theory; see \cite{FFhyperbolen2}. 
This concept seems to be related to our work, but we did not deal with deformations.
For this reason, it may be more appropriate a notion of honesty that depends only on the submanifold itself. 

\printbibliography

@article{DFTinter,
author = {Dajczer, M. and Florit, L. and Tojeiro, R.},
year = {1998},
%month = {01},
pages = {361-390},
title = {\textit{On deformable hypersurfaces in space forms}},
volume = {\textbf{174}},
journal = {Ann. Mat. Pura Appl.},
%doi = {10.1007/BF01759378}
}

@article{DFThyperEn2,
author = {Dajczer, M. and Florit, L. and Tojeiro, R.},
year = {2013},
%month = {10},
pages = {621-643},
title = {\textit{Euclidean hypersurfaces with genuine deformations in codimension two}},
volume = {\textbf{140}},
journal = {Manuscripta Math.},
%doi = {10.1007/s00229-012-0556-z}
}

@article{Moo,
author = {Moore, J.},
year = {1977},
%month = {06},
pages = {449–484},
title = {\textit{Submanifolds of constant positive curvature I}},
volume = {\textbf{44}},
journal = {Duke Math. J.},
%doi = {10.1215/S0012-7094-77-04421-0}
}

@article{YoSC,
author = {Guajardo, D.},
year = {2022},
%month = {06},
pages = {},
title = {\textit{Genuine deformations of Euclidean hypersurfacesin higher codimensions I}},
volume = {},
journal = {Arxiv},
%doi = {10.1215/S0012-7094-77-04421-0}
}

@article{FGsingular,
author = {Florit, L. and Guimarães, F.},
year = {2020},
%month = {03},
pages = {279–299},
title = {\textit{Singular genuine rigidity}},
volume = {\textbf{95}},
journal = {Comment. Math. Helv.},
%doi = {10.4171/CMH/488}
}

@article{DJgbendings,
  title={Genuine infinitesimal bendings of submanifolds},
  author={Dajczer, M. and Jimenez, M. I.},
  journal={preprint},
  year={2019}
}

@article{DTcompii,
author = {Dajczer, M. and Tojeiro, R.},
year = {1992},
%month = {09},
pages = {},
title = {\textit{On compositions of isometric immersions}},
volume = {\textbf{36}},
journal = {J. Diff. Geometry},
%doi = {10.4310/jdg/1214448440}
}

@article{DFgenrigcodim2,
author = {Dajczer, M. and Florit, L.},
year = {2004},
%month = {01},
pages = {195-210},
title = {\textit{Genuine Rigidity of Euclidean Submanifolds in Codimension Two}},
volume = {\textbf{106}},
journal = {Geom. Dedicata},
%doi = {10.1023/B:GEOM.0000033846.63094.46}
}

@article{FFhyperbolen2,
author = {Florit, L. and Freitas, G.},
year = {2017},
%month = {03},
pages = {751–797},
title = {\textit{Classification of codimension two deformations of rank two Riemannian manifolds}},
volume = {\textbf{25}},
journal = {Comm. Anal. Geom.}
%doi = {10.4310/CAG.2017.v25.n4.a2}
}

@article{YoFlatE,
author = {Guajardo, D.},
title = {\textit{Flat Euclidean submanifolds in high codimension}},
pages = {In preparation},
}

@article{DFGenDefSub,
author = {Dajczer, M. and Florit, L.},
year = {2004},
%month = {12},
pages = {1105–1129},
title = {\textit{Genuine Deformations of Submanifolds}},
volume = {\textbf{12}},
journal = {Comm. Anal. Geom.},
%doi = {10.4310/CAG.2004.v12.n5.a6}
}

@article {ChKineq,
AUTHOR = {Chern, S. and Kuiper, N. H.},
TITLE = {\textit{Some theorems on the isometric imbedding of compact Riemann
              manifolds in euclidean space}},
JOURNAL = {Ann. of Math.},
  %FJOURNAL = {Annals of Mathematics. Second Series},
    VOLUME = {\textbf{56}},
      YEAR = {1952},
     PAGES = {422-430},
      %ISSN = {0003-486X},
   %MRCLASS = {53.0X},
  %MRNUMBER = {50962},
%MRREVIEWER = {C. B. Allendoerfer},
       %DOI = {10.2307/1969650},
       %URL = {https://doi.org/10.2307/1969650},
}

@article{DGgaussP,
author = {Dajczer, M. and Gromoll, D.},
year = {1985},
%month = {01},
pages = {1–12},
title = {\textit{Gauss parametrizations and rigidity aspects of submanifolds}},
volume = {\textbf{22}},
journal = {J. Differential Geom.},
%doi = {10.4310/jdg/1214439717}
}

@article{DFaust,
author= {Dajczer, M. and Florit, L.},
title = {\textit{A class of austere submanifolds}},
year = {2001},
journal = {Illinois J. Math. },
volume = {\textbf{45}},
pages = {735-755},
}

@article{FZnegcurv2,
author = {Florit, L. and Zheng, F.},
year = {1999},
%month = {03},
pages = {1-15},
title = {\textit{On nonpositively curved euclidean submanifolds: splitting results II}},
volume = {\textbf{508}},
journal = {J. Reine Angew. Math.},
%doi = {10.1006/jmaa.2001.7810}
}

@article{FZnegcurv1,
author = {Florit, L. and Zheng, F.},
year = {1999},
%month = {03},
pages = {55-62},
title = {\textit{On nonpositively curved Euclidean submanifolds: splitting results}},
volume = {\textbf{74}},
journal = {Comment. Math. Helv.},
%doi = {10.1006/jmaa.2001.7810}
}

IMPA – Estrada Dona Castorina, 110

22460-320, Rio de Janeiro, Brazil

{\it E-mail address}: {\tt diego.navarro.g@ug.uchile.cl}
\end{document}